\pdfoutput=1

\documentclass[a4paper,10pt,a4wide]{article}
\usepackage[latin2]{inputenc}
\usepackage{amsmath}
\usepackage{amsfonts}
\usepackage{amssymb}
\usepackage{amsthm}
\usepackage{geometry}
\usepackage{hyperref}

\newtheorem{thm}{Theorem}[section]
\newtheorem{definition}[thm]{Definition}
\newtheorem{notation}[thm]{Notation}
\newtheorem{cor}[thm]{Corrollary}

\newtheorem{lem}[thm]{Lemma}
\newtheorem{obs}[thm]{Observation}

\newtheorem{quizz}[thm]{Question}
\newcommand{\pomega}{\mbox{$\mathcal{P}(\omega)$}}
\newcommand{\cont}{\mbox{$\mathfrak c$}}
\newcommand{\ba}{\mbox{${\mathbb B}$}}
\newcommand{\0}{\mbox{${\bf 0}$}}
\newcommand{\F}{\mbox{${\mathcal F}$}}

\newcommand{\cl}[1]{\mbox{$\overline{#1}$}}
\newcommand{\intr}[1]{\mbox{$\hbox{int} #1$}}

\title{Lonely points in $\omega^*$}
\author{Jonathan Verner\footnote{e-mail: jonathan.verner\@matfyz.cz, postal: Dvouletky 47, Praha 10 --- 100 00, Czech Republic} \\ KTIML Charles University, Prague}
\begin{document}
\date{July 2007}
\maketitle

\begin{abstract}
Inspired by the work of J. van Mill we define a new topological
type --- lonely points. We show that the question of whether these
points exist in $\omega^*$ is equivalent to finding a countable OHI,
extremally disconnected, zerodimensional space with a remote
weak P-point. We also present methods which allow us to find lonely points
in a large subspace of $\omega^*$ and show why known methods do not
allow us to construct them in all of $\omega^*$.

\vspace{0.2cm}
\end{abstract}

\section{Introduction and definitions}

The motivation for this paper comes from homogeneity questions in
topology.

\begin{definition} A topological space $X$ is \emph{homogeneous} if for any two
points there is a homeomorphism of the space mapping one to the other. A subset
$T\subset X$ of the space is a \emph{topological type} if it is invariant
under homeomorphisms, that is any homeomorphism maps points in $T$ to points in
$T$.
\end{definition}

It is immediately seen that if a space contains two distinct nonempty topological
types, then it cannot be homogeneous. This observation is generally used to
prove nonhomogeneity of topological spaces. Thus Rudin proved (\cite{Rudin})
that, assuming CH, $\omega^*$ is not homogeneous by proving that P-points
(which form a type) exist. A decade later Z. Frol\'\i k proved in ZFC
(\cite{Frolik}) that $\omega^*$ contains $2^{\cont}$ many distinct topological
types which gave a final answer to the homogeneity question for $\omega^*$.
His proof was, however, slightly unsatisfactory because it was of combinatorial
nature and did not give a ``topological'' description of even a single type.
This led to the question of whether one can prove the existence of a
``topologically'' defined type in $\omega^*$. It took another decade for the
first answer to appear. In 1978 Kunen proved (\cite{WeakPp}) the existence
of weak P-points in $\omega^*$:
\begin{definition} A point $p\in X$ is a \emph{weak P-point} if it is not a
limit point of a countable subset of $X$. A subset $P\subseteq X$ is a weak
P-set if it is disjoint from the closure of any countable set disjoint from it.
\end{definition}
Four years later, van Mill in his celebrated paper (\cite{Mill16}) gave a
``topological'' description of 16 topological types and proved (in ZFC) that
they exist in $\omega^*$. We give a definition of two points closely related to
van Mill's work which are relevant to this paper:
\begin{definition} A point $p\in X$ is
\begin{itemize}
 \item[(1)] \emph{Countably discretely untouchable} if it is not a limit point
             of a countable discrete set.
 \item[(2)] \emph{Uniquely accessible} if whenever it is a limit point
of two countable sets it is a limit point of their intersection.
\end{itemize}
\end{definition}

A natural question to ask is whether we can have both properties at the same
time. This leads to the following definition and questions:

\begin{definition}\label{lonelyDef} A point $p\in X$ is \emph{lonely} if it satisfies:
\begin{itemize}
\item[(i)] it is countably discretely untouchable,
\item[(ii)] it is a limit point of a countable crowded set and
\item[(iii)] it is uniquely accessible
\end{itemize}
\end{definition}
\begin{quizz} Does $\omega^*$ contain lonely points?
\end{quizz}
\begin{quizz} Is there at least a ``large'' subspace of $\omega^*$ which
contains lonely points?
\end{quizz}

It should be noted that the existence in $\omega^*$ of points satisfying
(i-ii) is due to Kunen under MA (\cite{SomePoints}) and in ZFC to
van Mill (\cite{Mill16}). However the methods of construction employed by
these authors cannot be used to construct a point satisfying (iii).
We comment on this in the last section.

The paper is organized as follows: In the second section we state some
embedding theorems, the third is devoted to the characterization of lonely
points, fourth introduces some facts about OHI spaces and the last section
gives a positive answer to the second question and some comments on the
potential solution to the first question. The first four sections are
preparatory material for the proof of the main theorem which is contained
in the last section.

We shall denote by $X^*$ the \v{C}ech-Stone remainder of $X$,
i.e. $\beta X\setminus X$, in particular $\omega^*$ is the space of free
ultrafilters on $\omega$ with the Stone topology. All the spaces considered
are, unless otherwise stated, assumed to be (at least) $T_{3\frac{1}{2}}$ and
crowded (that is without isolated points).

I would like to thank my advisor P. Simon for reading
the earlier versions of this paper and for his guidance while working on the
problem. I would also like to thank the referee for valuable comments which
helped to make the paper more readable.

\section{Finding points in $\omega^*$}

Working in $\omega^*$ is sometimes difficult and it can be easier to work
outside and then embed the resulting situation into $\omega^*$. This was a
technique van Mill used in his \cite{Mill16}. We use a theorem of Simon:

\begin{definition} A space is \emph{extremally disconnected} if the closure of
an open subset of the space is open.
\end{definition}

\begin{thm}[\cite{SimonIndep}]\label{simon} Every extremally disconnected compact space
 of weight $\leq\cont$ can be embedded into $\omega^*$ as a closed weak P-set.
\end{thm}

An easy observation shows that this theorem is adequate for our needs, since
the embedding preserves lonely points:

\begin{obs} If $p$ is a lonely point in $Y$ and if $Y$ is a weak P-set in $X$
then $p$ is a lonely point in $X$.
\end{obs}

Thus the plan is to build a countable extremally disconnected space $X$ such
that $\beta X$ contains a lonely point. Since $X$ is extremally disconnected iff
$\beta X$ is we can then use theorem \ref{simon} to transfer the point
into $\omega^*$.

\section{Characterization of lonely points}

In this section we take a closer look at the properties of lonely points.
We first introduce some definitions connected to these properties and list
some standard theorems for later reference. Then we prove a characterization
theorem for the existence of lonely points.

\begin{obs} Suppose $p\in\omega^*$ is lonely and $p\in\overline{S}\setminus S$, where S is countable and crowded. Then:
\begin{itemize}
\item[(a)] $p$ is not a limit point of any nowhere dense subset of $S$.
\item[(b)] $p$ is a weak P-point in $\overline{S}\setminus S$.
\item[(c)] Whenever $D, H$ are two dense subsets of $S$ then their
intersection is nonempty.
\end{itemize}
\end{obs}
\noindent

We shall see that these properties in fact characterize lonely points in $\omega^*$.
We now introduce the notions hinted at by the previous observation and some facts about them.

\subsection{Remote points}

\begin{definition}[\cite{Gillman}] A point $p\in X^*$ is a
\emph{remote point} of X if it is not a limit point (in $\beta X$) of a nowhere
dense subset of $X$. A closed filter $\F$ on $X$ is remote if for any nowhere
dense subset $N\subseteq X$ there is an $F\in\F$ which is disjoint from $N$.
\end{definition}

Remote points were investigated by several people (Chae and Smith, van Douwen
and others). Here we mention a theorem of van Douwen:

\begin{thm}[\cite{remote},5.2]\label{extDisc} $\beta X$ is extremally
disconnected at each remote point of $X$. As a corollary a remote point cannot
be in the closure of two disjoint open subsets of $X$.
\end{thm}

\subsection{Irresolvable spaces}

\begin{definition}[\cite{Hewitt}] A crowded topological space is
\emph{irresolvable} if it does not contain disjoint dense subsets, otherwise it
is resolvable. It is \emph{open hereditarily irresolvable} (\emph{OHI} for
short) if every open subspace is irresolvable.
\end{definition}

\begin{lem} \label{unionRES} The union of resolvable spaces is resolvable.
\end{lem}
\begin{proof} Let $\{X_\alpha:\alpha<\kappa\}$ be an enumeration of the resolvable topological spaces and for $\alpha<\kappa$ let $D_\alpha^i$, $i=0,1$ be
disjoint sets dense in $X_\alpha$. Define $D_{0}^{\prime i}=D_0^i$ and for
$i=0,1$ let
\begin{displaymath}
D_\alpha^{\prime i}=D_\alpha^i\setminus \overline{\bigcup_{\beta<\alpha}
D_\beta^{\prime i}}
\end{displaymath}
Let $D^i=\bigcup_{\alpha<\kappa}D_\alpha^{\prime i}$.
Since the $D_\alpha^{\prime i}$'s are dense in $X_\alpha$ necessarily
\begin{displaymath}
 \overline{\bigcup_{\beta<\alpha} D_\beta^{\prime 0}}=\overline{\bigcup_{\beta<\alpha} D_\beta^{\prime 1}}
\end{displaymath}
for every $\alpha<\kappa$ and we can conclude that $D^0$ is disjoint
from $D^1$. Both are dense in every $X_\alpha$ so also in their union, so
their union is resolvable.
\end{proof}

\begin{cor}\label{subOHI} Any irresolvable, not OHI space
contains a maximal (w.r.t. inclusion) resolvable subspace. Its complement is
open hereditarily irresolvable.
\end{cor}
\begin{proof}
Suppose $X$ is non OHI, there is an subset $A$ of $X$
which is resolvable. By the previous lemma the union $R$ of all resolvable
subspaces of $X$ containing $A$ is a resolvable proper (since $X$ is
irresolvable) subspace of $X$. Notice that $R$ must be closed because the
closure of a~resolvable space is resolvable. Suppose $B\subseteq X\setminus R$
is resolvable. Then $R\cup B$ is resolvable, a~contradiction with the definition
of $R$.
\end{proof}

\subsection{Characterization theorems}
We shall now make precise the statement from the beginning of this section that
certain properties in fact characterize lonely points in $\omega^*$:

\begin{thm}\label{equiv1} If $X$ is a countable OHI space with $p\in X^*$ a
remote point which is a weak P-point of $X^*$, then $p$ is a lonely point in
$\beta X$.
\end{thm}
\begin{proof} We first prove that the point is uniquely accessible.
Suppose $D_0,D_1\subseteq\beta X$ are two countable sets with
$p\in\cl{D_0}\cap\cl{D_1}$. Then, because $p$ is a weak P-point of $X^*$
$p\in\cl{D_0\cap X}\cap\cl{D_1\cap X}$. Because $p$ is remote,
$p\in\cl{\intr{\cl{D_0\cap X}}}\cap\cl{\intr{\cl{D_1\cap X}}}$. Again, because
$p$ is remote it cannot be in the closure of two disjoint open sets
(\ref{extDisc}), so $\intr{\cl{D_0\cap X}}\cap\intr{\cl{D_1\cap
X}}=G\neq\emptyset$, but now, since $X$ is OHI and $D_0,D_1$ are both dense in
$G$, we have that $D_0\cap D_1\neq\emptyset$.

An OHI space is crowded so condition (ii) of \ref{lonelyDef} is also satisfied and condition (i)
follows from the fact that discrete subsets of OHI spaces are nowhere dense (see [\cite{maximal}, 1.13]).
\end{proof}

On the other hand:

\begin{thm}\label{equiv2} If $p\in\omega^*$ is a lonely point then there is a
countable, extremally disconnected OHI space $X$ with a remote point which is a
weak P-point of $X^*$.
\end{thm}

\begin{proof}
 Let $X\subseteq \omega^*\setminus\{p\}$ be a countable set with $p\in\cl{X}$.
 Since $p$ is a lonely point $X$ must be irresolvable. By \ref{subOHI}
 we may assume $X$ is OHI (since $p$ cannot be in the closure of any resolvable
 subspace of $X$). Since $X$ is a countable subset of $\omega^*$ it is
 extremally disconnected. By (\cite{HST}, 1.5.2) it is $C^*$-embedded in
 $\omega^*$ so $\cl{X}\approx\beta X$. Because $p$ is a lonely point, it must
 be a weak P-point of $\cl{X}\setminus X\approx X^*$ and a remote point of $X$.
\end{proof}

Theorems \ref{equiv1}, \ref{equiv2} and \ref{simon} together give the
following characterization:

\begin{thm}\label{char1} There is a lonely point in $\omega^*$ iff
there is a countable extremally disconnected OHI space $X$ with a remote point
$p$ which is a weak P-point of $X^*$.
\end{thm}

From the proofs we see that the following is also true:

\begin{thm}\label{char2} If there is a countable extremally disconnected OHI
space $X$ with a remote point then there is a countable
set $S\subseteq\omega^*$ and $p\in\cl{S}$ such that
$(\omega^*\setminus (\cl{S}\setminus {S}))\cup\{p\}$ contains a lonely point.
\end{thm}

Theorem \ref{char2} will be used in the last section to answer the second question.

\section{Constructing OHI spaces}
In view of the characterization theorem we are interested in extremally
disconnnected OHI spaces. Here we present a tool for constructing such
spaces based on an idea of Hewitt (\cite{Hewitt}).

\begin{definition} If P is a property of a space we say that $(X,\tau)$
 is \emph{maximal P} if it has P and for any topology $\sigma$ finer than
$\tau$ the space $(X,\sigma)$ does not have P.
\end{definition}

We will need the following theorem:

\begin{thm}\label{compl} Every maximal crowded and zerodimensional space is
extremally disconnected.
\end{thm}
\begin{proof} It is enough to show that any regular open set must be clopen.
But if it were not, one could refine the topology by adding the complement of
the set. The finer topology would still be crowded and zerodimensional and this
would contradict maximality.
\end{proof}

Coupled with a theorem of Hewitt:

\begin{thm}[\cite{Hewitt}]\label{hewitt} Any maximal crowded and zerodimensional
topology is OHI.
\end{thm}

we get:

\begin{thm}\label{refine} Any zerodimensional crowded topology can be
refined to an OHI extremally disconnected, zerodimensional crowded topology.
\end{thm}

\section{The Main Theorem}

To prove the existence of lonely points in $\omega^*$ we would need an extremally
disconnected space with a remote point which is a weak P-point. Since the
weak P-point property is hard to achieve we want at least
remoteness to be able to use \ref{char2}. The theorem \ref{refine} from the previous
section suggests that we build a space with a remote closed filter and then refine the
topology to make the space extremally disconnected OHI. Unfortunately when refining the
topology new n.w.d. sets could appear and kill the remoteness of the filter. We
need a stronger version of remoteness:

\begin{definition} A closed filter $\F$ on $X$ is \emph{strongly remote} if for
any set $A\subseteq X$ with empty interior there is an $F\in\F$ which is
disjoint from $A$.
\end{definition}

It is easy to see that a strongly remote filter is a remote filter and also
that a strongly remote filter is strongly remote in any finer topology. This is
the key property, since if we build a strongly remote filter we can then use
the theorem \ref{refine} without loosing remoteness. This section will be
devoted to a single theorem which will give us a strongly remote filter on a
suitable space:

\begin{thm}\label{main}
There is a crowded, $T_2$, zerodimensional topology $\tau$ on
$\omega$ and a strongly remote filter on $\omega$ with this topology.
\end{thm}

\begin{proof}

The proof of the theorem will come in several steps. First, we state a standard
definition and lemma from boolean algebras.

\begin{definition} A boolean algebra $\ba$ has \emph{hereditary independence}
$\kappa$ if whenever $b\in\ba\setminus{\0}$ and $A\in[\{a:\0<a<b\}]^{<\kappa}$
there is an element
$c\in\ba$ with $\0<c<b$ which is independent from $A$; that is $a\wedge c\neq\0\neq a\setminus c$
for all $a\in A$.
\end{definition}

\begin{notation} If $I$ is an ideal on a boolean algebra, let $I^*=\{-b:b\in I\}$.
\end{notation}

\begin{lem} There is an ideal $I$ on $\omega$, extending $FIN$ and such that
$\pomega/I$ has hereditary independence $\cont$.
\end{lem}
\begin{proof} The complete Boolean algebra $\ba=Compl(Clopen(2^{\cont}))$ has
hereditary independence $\cont$ and is $\sigma$-centered so there
is an ideal $I$ on $\omega$ such that $\ba$ is isomorphic to $\pomega/I$.
\end{proof}

By the previous lemma fix an ideal $I\supseteq FIN$ on $\omega$ such
that $\pomega/I$ has hereditary independence $\cont$. Throughout this proof we
will adopt the following notation:

\newcommand{\CO}{\mbox{${\mathcal CO}$}}

\begin{notation}
 If $I$ is an ideal on $\omega$ and $\CO$ is a system of subsets of $\omega$ let
$\tau_I(\CO)$ denote the topology generated by $\{U,\omega\setminus
U:U\in\CO\}\cup I^*$.
\end{notation}

Now let $\langle A_\alpha:\omega\leq\alpha<\cont,\alpha\ \hbox{even}\rangle$ be
an enumeration of $\pomega$,
$\langle (G_\alpha,n_\alpha):\omega\leq\alpha<\cont,\alpha\ \hbox{odd}\rangle$
an enumeration of $\{(G, n): G\in I^*, n\in G\}$ and
$\langle K_n:n<\omega\rangle$ an enumeration of $[\omega]^2$. Let
$F_0=\F_0=\CO_0=\emptyset$.

Proceed by induction constructing $F_\alpha$ (a closed filterbase), $\F_\alpha$
(a closed filter), $\CO_\alpha$ (clopen sets), $\tau_\alpha$ (topology) for
$\alpha<\cont$ such that the following is satisfied:
\begin{itemize}
 \item [(i)] $|\CO_\alpha|\cdot|F_\alpha|\leq \alpha\cdot\omega$
  for each  $\alpha<\cont$.
 \item [(ii)] If $\alpha<\cont$ is limit, then
  $F_\alpha=\bigcup_{\beta<\alpha}F_\beta$, $\CO_\alpha=\bigcup_{\beta<\alpha}
  \CO_\beta$.
 \item [(iii)] $\tau_\alpha=\tau_I(\CO_\alpha)$, $\F_\alpha$ is the
  filter generated by $F_\alpha$, $F_\alpha\subseteq\CO_\alpha$ for
  $\alpha<\cont$
 \item [(iv)] The family $\{[U]_I:U\in\CO_\alpha\}$ is independent in
  $\pomega/I$ for each $\alpha<\cont$. (to make $\tau$ crowded).
 \item [(v)] For each $n<\omega$ there is an $U\in\CO_{n+1}$ such that
  $|U\cap K_n| = 1$ (to make $\tau\ T_2$).
 \item [(vii)] If $\omega<\alpha<\cont$ is odd then there is
  $U\in\CO_{\alpha+1}$ with  $n\in U\subseteq G_\alpha$ (to make $\tau$
  zerodimensional).
 \item [(viii)] If $\omega<\alpha<\cont$ is even then either
  $\hbox{int}_{\tau_{\alpha+1}}A_\alpha\neq\emptyset$ or there
  is $F\in\F_{\alpha+1}$ which misses $A_\alpha$ (to make $\F$ strongly remote).
\end{itemize}

Suppose, that the construction can indeed be carried out. Let $\F$ be the filter
generated by $\bigcup\{F_\alpha:\alpha<\cont\}$,
$\CO=\bigcup\{CO_\alpha:\alpha<\cont\}$ and
$\tau=\tau_I(\CO)$. Then $(\omega,\tau)$ with
the closed filter $\F$ statisfy the conclusion of the proof:

The topology is zerodimensional (the definition takes care of the sets from
$\CO$, condition (vii) takes care of the sets from $I^*$).

The topology is also $T_2$ because if $x\neq y\in\omega$ then there is
$n<\omega$, such that $K_n=\{x,y\}$ and by (v) there is
$U\in\CO_n\subseteq\CO$ such that $|U\cap K_n|=1$. This $U$ is $\tau$-clopen and
separates $x$ from~$y$.

To show that $\tau$ is crowded it is sufficient to consider its basis, which
consists of elements of the form:
\begin{equation}\label{basElem}
 \bigcap_{U\in P}U \cap \bigcap_{V\in N}(\omega\setminus V) \cap G
\end{equation}
where $P,N\in[\CO]^{<\omega}, G\in I^*$. Now, by (iv) the family
$\{[U]_I:U\in\CO\}$ is
independent in $\pomega/I$ with $FIN\subseteq I$ and $G\in I^*$
so (\ref{basElem}) is finite iff there is some $U\in N\cap P$. But then
(\ref{basElem}) must be a subset of $U\cap (\omega\setminus U)$ so it must be
empty. Thus the basis does not contain any finite sets beyond the empty set, so
it is crowded as is the whole topology.

To prove that $\F$ is strongly remote, choose $O\subseteq\omega$ such that
$\hbox{int}_\tau O=\emptyset$. There is an $\alpha<\cont$, such that
$O=A_\alpha$. Then $\hbox{int}_{\tau_{\alpha+1}} A_\alpha=\emptyset$,
so there is $F\in\F_{\alpha+1}\subseteq\F$ such that $F\cap A_\alpha
=\emptyset$.

So it remains to be shown that the inductive construction can be carried out
all the way up to $\cont$. Suppose that we are at stage $\alpha<\cont$.
If $\alpha$ is limit, we can let $\F_\alpha=\bigcup\{\F_\beta:\beta<\alpha\}$
and the conditions will be satisfied. Otherwise $\alpha = \beta+1$. There are
three cases:

\paragraph{Case $\beta = n<\omega$.} Let $K_n=\{x,y\}$. Then by (iv) the
subset $\{[U]:U\in \CO_n\}$ of $\pomega/I$ is independent. Since $\pomega/I$
has independence $\cont$ and since $|\CO_n|\leq\omega<\cont$, there is
an $U^\prime\in \pomega$ such that $\{[U]_I:U\in\CO_n\}\cup\{[U^\prime]_I\}$ is
still independent. Then let $U=(U^\prime\cup\{x\})\setminus\{y\}$. We have that
$[U^\prime]_I=[U]_I$ so (i,iv,v) are satisfied if we let
$\CO_\alpha=\CO_n\cup\{U\}$.

\paragraph{Case $\omega<\beta<\cont,\ \beta$ odd.}
Since $|\CO_\beta|\leq\beta<\cont$ and because $\pomega/I$ has hereditary
independence $\cont$ there is $U^\prime\subseteq\omega$ such that
$\CO_\beta/I\cup\{[U^\prime]_I\}$ is still independent. Let $U=(U^\prime\cap
G_\beta)\cup\{n_\beta\}$. Because $\{n_\beta\}\in I,G_\beta\in I^*$
we have that $[U^\prime]_I=[U]_I$ so we can let $\CO_\alpha=\CO_\beta\cup\{U\}$
and (vii) with all other conditions is satisfied.

\paragraph{Case $\omega<\beta<\cont,\ \beta$ even.}
If $\{[U]_I:U\in\CO_\beta\}\cup\{[\omega\setminus A_\beta]_I\}$ is
independent in $\pomega/I$, then we can let
$\CO_\alpha=\CO_\beta\cup\{\omega\setminus A_\beta\}$ and again all conditions
are satisfied. So suppose otherwise.

If we let $B=\omega\setminus A_\beta$, necessarily $B\not\in I$ (otherwise
already $\hbox{int}_{\tau_0} A_\beta\neq\emptyset$). We claim, that
$\{[U\cap B]_I : U\in\CO_\beta\}$ is independent in $\pomega/I\upharpoonright
[B]_I$: If it were not, then for some elementary meet $M$ over $\CO_\beta$
we would have that $M\cap B\in I$ but then, since $M\subseteq_I A_\beta$,
$\hbox{int}_{\tau_\beta} A_\beta\neq\emptyset$ a contradiction. Now, since
$\pomega/I$ has hereditary independence $\cont$, $\{[U\cap~B]_I:U\in\CO_\beta\}$
is not maximal independent in $\pomega/I\upharpoonright [B]_I$ (by (i)
$|\CO_\beta|\leq\beta<\cont$), so there is $F\subseteq B$ such that $\{[F\cap
B]_I:U\in\CO_\beta\}\cup\{[F]_I\}$ is independent in $\pomega/I\upharpoonright
[B]_I$ so, a fortiori, $\{[U]_I:U\in\CO_\beta\}\cup\{F\}$ is independent in
$\pomega/I$ and if we let $F_\alpha=F_\beta\cup\{F\}$ and
$\CO_\alpha=\CO_\beta\cup\{F\}$ all conditions are satisfied and we are done.
\end{proof}

Together the theorems \ref{refine}, \ref{char2} and \ref{main} give
us the following:

\begin{thm} There is a countable set $S\subseteq\omega^*$ and $p\in\cl{S}$ such
that $(\omega^*\setminus (\cl{S}\setminus S))\cup\{p\}$ contains a lonely point.
\end{thm}

To get a lonely point in $\omega^*$ we would need the filter from \ref{main} to
be a weak P-point. However the only suitable constructions of weak P-points
yield variants of OK points. These points cannot be limit points of ccc
sets. Countable OHI spaces are nowhere locally compact, so their remainder is
ccc and therefore these constructions are of no use. This leads to the following question:

\begin{quizz} \label{wppQ} Is there a countable (or at least separable) nowhere locally compact
space $X$ with a weak P-point of $X^*$.
\end{quizz}

Also the main question remains open:

\begin{quizz} Is there a lonely point in $\omega^*$?
\end{quizz}

\bibliography{lonely-points}
\end{document}